\documentclass[a4paper, 11pt]{article}

\usepackage{amsmath}
\usepackage{amsfonts}
\usepackage{amssymb}
\usepackage[english]{babel}
\usepackage{graphicx}
\usepackage{amsthm}
\usepackage{hyperref}

\newtheorem{proposition}{Proposition}
\newtheorem{theorem}{Theorem}
\newtheorem{lemma}{Lemma}

\newtheorem*{cor4.1}{Corollary 4.1}
\newtheorem*{th2.5}{Theorem 2.5}

\theoremstyle{definition}

\begin{document}

\title{Indecomposable $1$-factorizations of the complete multigraph
$\lambda K_{2n}$ for every $\lambda\leq 2n$\thanks{Research performed within the activity of
INdAM--GNSAGA with the financial support of the Italian Ministry
MIUR, project ``Strutture Geometriche, Combinatoria e loro
Applicazioni''}
}

\author{S. Bonvicini \thanks{Dipartimento di Scienze
Fisiche, Informatiche e Matematiche, Universit\`a di Modena e Reggio
Emilia, via Campi 213/b, 41126 Modena (Italy)}\,,
G. Rinaldi\thanks{Dipartimento di Scienze e Metodi dell'Ingegneria,
Universit\`a di Modena e Reggio Emilia, via Amendola 2, 42122 Reggio
Emilia (Italy)}}

\maketitle

\begin{abstract}
\noindent A  $1$-factorization of the complete multigraph $\lambda K_{2n}$
is said to be indecomposable if it cannot be represented as the union
of $1$-factorizations of $\lambda_0 K_{2n}$ and $(\lambda-\lambda_0) K_{2n}$,
where $\lambda_0<\lambda$. It is said to be simple if no $1$-factor is repeated.
For every $n\geq 9$ and for every $(n-2)/3\leq\lambda\leq 2n$, we construct an indecomposable
$1$-factorization of $\lambda K_{2n}$ which is not simple. These $1$-factorizations provide  simple and
indecomposable $1$-factorizations of $\lambda K_{2s}$ for every $s\geq 18$ and
$2\leq\lambda\leq 2\lfloor s/2\rfloor-1$. We also give a generalization of a result
by Colbourn et al. which provides a simple and indecomposable $1$-factorization
of $\lambda K_{2n}$, where $2n=p^m+1$, $\lambda=(p^m-1)/2$, $p$ prime.

\end{abstract}



\section{Introduction}\label{sec:intro}

We refer to \cite{BonMur} for graph theory notation and terminology which are not introduced
explicitly here. We recall that the complete multigraph $\lambda K_{2n}$ has $2n$ vertices and each pair of vertices
is joined by exactly $\lambda$ edges. A $1$-factor of $\lambda K_{2n}$ is a spanning
subgraph of $\lambda K_{2n}$ consisting of $n$ edges that are pairwise independent.
If $\mathcal S$ is a set of $1$-factors of $\lambda K_{2n}$, then we will denote by 
$E(\mathcal S)$ the multiset containing all the edges of the $1$-factors of $\mathcal S$, namely, 
$E(\mathcal S)=\cup_{F\in\mathcal S}\, E(F)$.
A $1$-factorization $\mathcal F$ of $\lambda K_{2n}$ is a partition of the edge-set of $\lambda K_{2n}$
into $1$-factors. A subfactorization of $\mathcal F$ is a subset
$\mathcal F_0$ of $1$-factors belonging to $\mathcal F$ that constitute a
$1$-factorization of $\lambda_0 K_{2n}$, where $\lambda_0\leq\lambda$. For every $\lambda\geq 1$, it is possible to find a
$1$-factorization of $\lambda K_{2n}$. Lucas' construction provides a $1$-factorization
for the complete graph $K_{2n}$, denoted by $GK_{2n}$ (see \cite{Lu}). By taking $\lambda$ copies of
$GK_{2n}$, we find a $1$-factorization of $\lambda K_{2n}$. Obviously, it 
contains repeated $1$-factors. Moreover, we can consider $\lambda_0<\lambda$
copies of each $1$-factor so that it is the union of
$1$-factorizations of $\lambda_0 K_{2n}$ and $(\lambda-\lambda_0)K_{2n}$.
A $1$-factorization of $\lambda K_{2n}$ that contains no repeated $1$-factors is said to be \emph{simple}.
A $1$-factorization of $\lambda K_{2n}$ that can be represented as the union of
$1$-factorizations of $\lambda_0 K_{2n}$ and $(\lambda-\lambda_0) K_{2n}$, where
$\lambda_0<\lambda$, is said to be \emph{decomposable}, otherwise it is called \emph{indecomposable}.
An indecomposable $1$-factorization might be simple or not.
In this paper, we consider the problem about the existence of indecomposable
$1$-factorizations of $\lambda K_{2n}$. Obviously, $\lambda> 1$. In order that the complete multigraph $\lambda K_{2n}$
admits an indecomposable $1$-factorization, the parameter $\lambda$ cannot be arbitrarily large: 
we have necessarily $\lambda< 3\cdot 4\cdots (2n-3)$ or $\lambda<[n(2n-1)]^{n(2n-1)}\binom{2n^3+n^2-n+1}{2n^2-n}$,
according to whether the $1$-factorization is simple or not (see \cite{BaarWall}).
Moreover, two non-existence results are known. For every $\lambda >1$ there is no indecomposable
$1$-factorization of $\lambda K_4$ (see \cite{CCR}). For every $\lambda \geq 3$ there is no indecomposable
$1$-factorization of $\lambda K_6$ (see \cite{BaarWall}). We recall that in \cite{CCR} the authors construct simple and indecomposable
$1$-factorizations of $\lambda K_{2n}$ for $2\leq\lambda\leq 12$,
$\lambda\neq 7, 11$. They also give a simple and indecomposable $1$-factorization of
$\lambda K_{p+1}$, where $p$ is an odd prime and $\lambda=(p-1)/2$.
In \cite{ArchDinitz} we can find an indecomposable $1$-factorization of $(n-p)K_{2n}$,
where $p$ is the smallest prime not dividing $n$. This $1$-factorization
is not simple, but it is used to construct a simple and indecomposable $1$-factorization of $(n-p)K_{2s}$
for every $s\geq 2n$. This construction improves the results in \cite{CCR} for $2\leq\lambda\leq 12$
(see Theorem $2.5$ in \cite{ArchDinitz}).
Simple and indecomposable $1$-factorizations of $(n-d)K_{2n}$, with $d\geq 2$, $n-d\geq 5$ and $\gcd(n, d)=1$,
are constructed in \cite{Chu}. Other values of $\lambda$ and $n$ for which the existence of a simple
and indecomposable $1$-factorization of $\lambda K_{2n}$ is known are the following:
$2n=q^2+1$, $\lambda=q-1$, where $q$ is an odd prime power (see \cite{KSS});
$2n=2^h+2$, $\lambda=2$ (see \cite{Son});
$2n=q^2+1$, $\lambda=q+1$, where $q$ is an odd prime power (see \cite{Kiss});
$2n=q^2$, $\lambda=q$, where $q$ is an even prime power (see \cite{Kiss}).

In this paper we prove some theorems about the
existence of simple and indecomposable $1$-factorizations of $\lambda K_{2n}$,
where most of the parameters $\lambda$ and $n$ were not previously considered in literature.
We show that for every $n\geq 9$ and for every $(n-2)/3\leq\lambda\leq 2n$ there exists an indecomposable
$1$-factorization of $\lambda K_{2n}$ (see Theorem \ref{th1}). We can also
exhibit some examples of indecomposable $1$-factorizations of $\lambda K_{2n}$ for $n\in\{7, 8\}$,
$(n-2)/3\leq\lambda\leq n$ (see Proposition \ref{pro4}); and for $n\in\{5, 6\}$,
$(n-2)/3\leq\lambda\leq n-2$ (see Proposition \ref{pro1} and \ref{pro2}). The $1$-factorizations
in Theorem \ref{th1}, Proposition \ref{pro1}, \ref{pro2} and \ref{pro4}
are not simple. By an embedding result in \cite{CCR}, we can use them to prove the existence of
simple and indecomposable $1$-factorizations of $\lambda K_{2s}$ for every $s\geq 18$ and
for every $2\leq\lambda\leq 2\lfloor s/2\rfloor-1$ (see Theorem \ref{th2}). We note that for odd values of $s$,
the parameter $\lambda$ does not exceed the value $s-2$. Nevertheless, if
$2s=p^m+1$, where $p$ is a prime, then we can find a simple and indecomposable $1$-factorization of $(s-1) K_{2s}$
(see Theorem \ref{th3}).
By our results we can improve Theorem $2.5$ in
\cite{ArchDinitz} about the existence of simple and indecomposable $1$-factorizations of
$\lambda K_{2n}$ for $2\leq\lambda\leq 12$. We note that in Theorem 2.5 in \cite{ArchDinitz}
the existence of a simple and indecomposable $1$-factorization of $11 K_{2n}$ (respectively, $12 K_{2n}$)
is known for every $2n\geq 52$ (respectively, $2n\geq 32$). By Theorem \ref{th2}, a simple and indecomposable
$1$-factorization of $11 K_{2n}$ exists for every $2n\geq 36$.
By Theorem \ref{th3}, there exists a simple and indecomposable
$1$-factorization of $12 K_{26}$. Moreover, Theorem \ref{th3} extends Theorem $2$ in \cite{CCR}
to each odd prime power.

\section{Basic lemmas.}\label{sec:set_1factors}

In Section \ref{sec:IOF_nosimple} and \ref{sec:IOF_simple} we will construct
indecomposable $1$-factorizations of $\lambda K_{2n}$ for suitable
values of $\lambda> 1$. These $1$-factorizations contain $1$-factor-orbits, that is, sets of $1$-factors belonging to the
same orbit with respect to a group $G$ of permutations on the vertices of the complete multigraph.

If not differently specified, we use the exponential notation for the action
of $G$ and its subgroups on vertices, edges and $1$-factors. So, if $e=[x, y]$
is an edge of $\lambda K_{2n}$ and $g\in G$ we set $e^g=[x^g, y^g]$.
Analogously, if $F$ is a $1$-factor we set $F^g=\{e^g : e\in F\}$.
Since we shall treat with sets and multisets, we specify that by an edge-orbit $e^H$, where $H\leq G$,
we mean the set $e^H=\{e^h : h\in H\}$ and by a $1$-factor-orbit $F^H$ we
mean the set $F^H=\{F^h : h\in H\}$. If $h\in H$ leaves $F$ invariant, that is,
$F^h=F$, then $h$ is an element of the stabilizer of $F$ in $G$, which will be denoted by $G_F$.
The cardinality of $F^H$ is $|H|/|H\cap G_F|$. The following result holds.

\begin{lemma}\label{lemma1}
Let $F$ be a $1$-factor of $\lambda K_{2n}$ containing exactly $\mu$ edges belonging to
the same edge-orbit $e^H$, where $H$ is a subgroup of $G$ having trivial intersection with
the stabilizer of $F$ in $G$ and with the stabilizer of $e$ in $G$. 
The multiset $\cup_{h\in H}\, E(F^h)$ contains every edge of $e^H$ exactly $\mu$ times.
\end{lemma}

\begin{proof} We denote by $e_1,\ldots, e_{\mu}$ the edges in $F\cap e^H$. 
We show that every edge $f\in e^H$ appears $t_f\geq\mu$ times in the multiset
$E(F^H)=\cup_{h\in H}\, E(F^h)$. For every edge $e_i\in\{e_1,\ldots, e_{\mu}\}$ there exists an element
$h_i\in H$ such that $e^{h_i}_i=f$, since $e_i$ and $f$ belong to the same edge-orbit $e^H$.
Hence the $1$-factor $F^{h_i}$ contains the edge $f$. The $1$-factors $F^{h_1}$, $F^{h_2},\ldots F^{h_{\mu}}$ 
are pairwise distinct, since $H$ has trivial intersection with $G_F$. Therefore, every edge
$f\in e^H$ appears $t_f\geq\mu$ times in the multiset $E(F^H)$. We prove that $t_f=\mu$.
In fact, $t_f>\mu$ implies the existence of $h\in H\smallsetminus\{h_1,\ldots, h_{\mu}\}$
such that $f\in F^h$ and then $e^{h_i}_i=f=e^h_i$ for some $e_i\in\{e_1,\ldots, e_{\mu}\}$.
That yields a contradiction, since $e_i$, as well as $e$, has trivial stabilizer in $H$. 
\end{proof}

To prove the indecomposability of the $1$-factorizations in Section
\ref{sec:IOF_nosimple}, we will use the following lemma.

\begin{lemma}\label{lemma2}
Let $M$ be a $1$-factor of $\lambda K_{2n}$.
Let $\mathcal F$ be a $1$-factorization of $\lambda K_{2n}$
containing $0\leq\lambda-t<\lambda$ copies of $M$ and a subset $\mathcal S$ of $1$-factors
satisfying the following properties:

\begin{description}
	\item [(i)] the multiset $E(\mathcal S)$ contains every edge of $M$ exactly $t$ times;
	\item [(ii)] for every $\mathcal S'\subset\mathcal S$, the multiset $E(\mathcal S')$ contains
	$0<\mu< n$ distinct edges of $M$.
\end{description}

If $\mathcal F_0\subseteq\mathcal F$ is a $1$-factorization of $\lambda_0 K_{2n}$,
where $\lambda_0\leq\lambda$, then $\mathcal S\subseteq\mathcal F_0$ or $\mathcal F_0$
contains no $1$-factor of $\mathcal S$.
\end{lemma}

\begin{proof} Assume that $\mathcal F_0$ contains $0< s< |\mathcal S|$ elements of $\mathcal S$, say $F_1,\ldots F_s$.
We denote by $M'$ the set consisting of the edges of $M$ that are contained in the multiset $\cup^s_{i=1} E(F_i)$.
By property $(ii)$, the set $M'$ is a non-empty proper subset of $M$.
It is clear from $(i)$ that the $1$-factors of $\mathcal F$ containing some edges of $M$
are exactly the $\lambda-t$ copies of $M$ together with the $1$-factors of $\mathcal S$.
Therefore, the $1$-factorization $\mathcal F_0$ contains $\lambda_0$ copies of $M$, since
the edges in $M\smallsetminus M'$ are not contained in $\cup^{s}_{i=1} E(F_i)$.
Then the multiset $E(\mathcal F_0)$ contains at least $\lambda_0+1$ copies
of each edge in $M'$, a contradiction. Hence $s=n$
or $\mathcal F_0$ contains no $1$-factor of $\mathcal S$.\end{proof}

\section{Indecomposable $1$-factorizations which are not simple.}\label{sec:IOF_nosimple}

In what follows, we consider the group $G$ given by the direct product $\mathbb Z_n\times\mathbb Z_2$
and denote by $H$ the subgroup of $G$ isomorphic to $\mathbb Z_n$.
We will identify the vertices of the complete multigraph $\lambda K_{2n}$ with the
elements of $G$, thus obtaining the graph $\lambda K_{G}=\left (G, \lambda\binom G2\right)$,
where $\binom G2$ is the set of all possible $2$-subsets of $G$
and $\lambda\binom G2$ is the multiset consisting of $\lambda$ copies of $\binom G2$.

In $G$ we will adopt the additive notation and observe that $G$ is a group
of permutations on the vertex-set, that is, each $g\in G$ is identified with the permutation
$x\to x+g$, for every $x\in G$. For the sake of simplicity,
we will represent the elements of $G$ in the form
$a_j$, where $a$ and $j$ are integers modulo $n$ and modulo $2$, respectively.
The edges of $\lambda K_G$ are of type $[a_0, b_1]$ or $[a_j, b_j]$ and
we can observe that each edge $[a_0, b_1]$ has trivial stabilizer in $H$.
For every $a\in\mathbb Z_n$, we consider the edge-orbit
$M_a=[0_0, a_1]^H$. Each edge-orbit $M_{a}$ is a $1$-factor of $\lambda K_G$.
The $1$-factors in $\cup_{a\in\mathbb Z_n} M_a$ partition the edges of type $[a_0, b_1]$.
We shall represent the vertices and the $1$-factors $M_{a}$ as in Figure \ref{fig1}.
Observe that, if $M_a$ contains the edge $[x_0, (x+a)_1]$, then $M_{n-a}$ contains
the edge $[(x+a)_0, x_1]$.

The edges of type $[a_j, b_j]$, with $j=0, 1$, can be partitioned by the $1$-factors (or, near $1$-factors) of a
$1$-factorization (or, of a near $1$-factorization) of $K_n$. More specifically, for even values of $n$ we
consider the well-known $1$-factorization $GK_n$ defined by Lucas \cite{Lu}. We recall that in $GK_n$
the vertex-set of $K_n$ is $\mathbb Z_{n-1}\cup\{\infty\}$ and $GK_n=\{L_i : i\in\mathbb Z_{n-1}\}$,
where $L_0=\{[a, -a]: a\in\mathbb Z_{n-1}-\{0\}\}\cup\{[0,\infty]\}$ and
$L_i=L_0+i=\{[a+i, -a+i]: a\in\mathbb Z_{n-1}-\{0\}\}\cup\{[i,\infty]\}$. For odd values of $n$,
we consider the $1$-factorization $GK_{n+1}$ and delete the vertex $\infty$. Each
$1$-factor $L_i$ yields a near $1$-factor $L^*_i$ of $K_n$ where the vertex $i\in\mathbb Z_n$ in unmatched.
We denote by $GK^*_n$ the resulting near $1$-factorization of $K_n$.

For even values of $n$, we partition the edges $[a_j, b_j]$ of $\lambda K_{2n}$ into $1$-factors
of $\lambda K_{2n}$ as follows. For $j=0, 1$, we consider the $1$-factorization $GK_n$ of the complete graph $K_n$
with vertex-set $V_j=\{a_j: 0\leq a\leq n-1\}$. It is possible to obtain
a $1$-factor of $K_{2n}$ by joining, in an arbitrary way, a $1$-factor on $V_0$ to a $1$-factor on $V_1$.
We denote by $\mathcal F(GK_n)$ the resulting set of $1$-factors of $K_{2n}$. We denote by $\mathcal F(\lambda GK_n)$
the multiset consisting of $\lambda$ copies of $\mathcal F(GK_n)$.

For odd values of $n$, we partition the edges $[a_j, b_j]$ of $\lambda K_{2n}$ into $1$-factors
of $\lambda K_{2n}$ as follows. For $j=0, 1$, we consider the near $1$-factorization $GK^*_n$ of
the complete graph $K_n$ with vertex-set $V_j=\{a_j: 0\leq a\leq n-1\}$.
We select an integer $b\in\mathbb Z_n$. For $i=0,\ldots, n-1$,
we join the near $1$-factor $L^*_i$ on $V_0$ to the near $1$-factor $L^*_{i+b}$ on $V_1$
(subscripts are considered modulo $n$) and add the edge $[i_0, (i+b)_1]$. We obtain a $1$-factor
of $K_{2n}$. We denote by $\mathcal F(GK^*_n, b)$ the resulting set of $1$-factors of $K_{2n}$.
We denote by $\mathcal F(\lambda GK_n, b)$ the multiset consisting of $\lambda$ copies of $\mathcal F(GK_n, b)$.
Observe that the set $\{[i_0, (i+b)_1]: 0\leq i\leq n-1\}$ corresponds to the $1$-factor $M_b$.
Hence $\mathcal F(\lambda GK_n, b)$ contains every edge of $M_b$ exactly $\lambda$ times.

\begin{figure*}
\begin{center}
	\includegraphics[width=12cm]{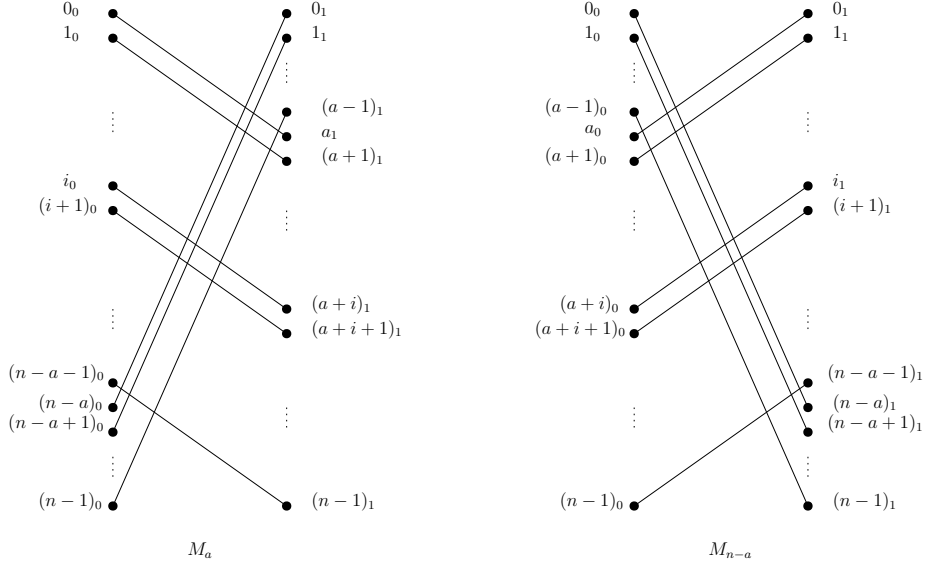}
\end{center}
	\caption{The vertices $a_0$, $b_1$ of $\lambda K_G$ are represented on the
	left and on the right, respectively. Each edge-orbit $M_a$ is a $1$-factor
	of $\lambda K_G$. If $M_a$ contains the edge $[0_0, a_1]$, then $M_{n-a}$ contains
	the edge $[a_0, 0_1]$.}
	\label{fig1}
\end{figure*}

In the following propositions we will construct $1$-factorizations of
$\lambda K_G$ which are not simple. They are obtained as described in Lemma
\ref{lemma3}. Moreover, Lemma \ref{lemma4} will be usefull to prove that these
$1$-factorizations are indecomposable. It is straightforward to prove that
the following holds.

\begin{lemma}\label{lemma3}
Let ${\cal F}'=\{F_1,\dots ,F_m\}$ be a set of $1-$factors of $\lambda K_G$ such that each $F_i$ contains no edge of
type $[a_j, b_j]$, has trivial stabilizer in $H$ and $F_r\notin F_i^H$  for each pair $(i,r)$ with $i\ne r$.

Let  ${\cal M}$ be the subset of $\{M_a :  a\in \mathbb Z_n\}$ containing all the
$1-$factors $M_a$ such that $t(M_a)=\sum_{i=1}^{m}|E(M_a)\cap E(F_i)| > 0$.

If $|H|=n$ is even and $t(M_a)\le \lambda$ for every $M_a\in {\cal M}$, then there exists a
$1-$factorization of $\lambda K_G$ whose $1-$factors are exactly those of 
$F_1^H\cup \dots \cup F_m^H \cup {\cal F}(\lambda GK_n)$
together with $\lambda - t(M_a)$ copies of each $M_a\in {\cal M}$ and $\lambda$ copies of each $M_a\notin {\cal M}$.

If $|H|=n$ is odd, $t(M_a)\le \lambda$ for every $M_a\in {\cal M}$ and there exists at least one
$1-$factor $M_b\in  \{M_a \ : \ a\in \mathbb Z_n\} \smallsetminus {\cal M}$, then there exists a $1-$factorization
of $\lambda K_G$ whose $1-$factors are exactly those of $F_1^H\cup \dots \cup F_m^H\cup {\cal F}(\lambda GK_n, b)$
together with $\lambda - t(M_a)$ copies of each $M_a\in {\cal M}$ and $\lambda$ copies of each $M_a\notin {\cal M}\cup \{M_b\}$.
\qed
\end{lemma}



\begin{lemma}\label{lemma4}
Let ${\cal F}$ be the $1-$factorization of $\lambda K_G$ obtained in Lemma \ref{lemma3}
starting from ${\cal F}'=\{F_1, \dots ,F_m\}$ and the set ${\cal M}$.
Let $\mathcal F_0\subseteq\mathcal F$ be a $1$-factorization of $\lambda_0 K_G$, $\lambda_0\leq\lambda$.
Let $F_i\in\mathcal F'$ and $M_a\in\mathcal M$ be such that $F_i$ contains
exactly one edge of $M_a$. If one of the following conditions holds:

\begin{description}
	\item [(i)] each $1$-factor in $\mathcal F'\smallsetminus\{F_i\}$ contains no edge of $M_a$;
	
 \item [(ii)] each $1$-factor  $F\in\mathcal F'\smallsetminus\{F_i\}$ containing some edge of $M_a$
	is such that either $F^H \subset {\cal F}_0$ or $F^H\cap {\cal F}_0= \emptyset$.
\end{description}

then it is either $F_i^H \subset {\cal F}_0$ or $F_i^H \cap {\cal F}_0=\emptyset$.
\end{lemma}

\begin{proof} Assume that $F_i$ satisfies property $(i)$. By Lemma \ref{lemma1},
each edge of $M_a$ appears exactly once in the multiset $E(F^H_i)$.
Since each $1$-factor in $\mathcal F'\smallsetminus\{F_i\}$ contains no edge of $M_a$,
the $1$-factorization $\mathcal F$ contains exactly $\lambda-1$ copies of $M_a$.
The assertion follows from Lemma \ref{lemma2} by setting $\mathcal S=F_i^H$ and $M=M_a$.

Assume that $F_i$ satisfies property $(ii)$. We can consider the subset $\mathcal F_1$ of 
$\mathcal F'\smallsetminus\{F_i\}$ consisting
of the $1$-factors $F$ containing $s_F\geq 1$ edges of $M_a$ and whose orbit $F^H$
is contained in $\mathcal F_0$. The set $\mathcal F_1$ might be empty. By Lemma \ref{lemma1},
each edge of $M_a$ appears exactly $s_F\geq 1$ times in the multiset $E(F^H)$,
where $F\in\mathcal F_1$. Hence $\lambda_0\geq\sum_{F\in\mathcal F_1} s_F\geq 0$ (if $\mathcal F_1=\emptyset$,
then $\sum_{F\in\mathcal F_1} s_F=0$). Set ${\cal S}=F_i^H\cap {\cal F}_0 $ and suppose that
$0 < |\mathcal S| < n$, where $n=|F^H_i|$. Let $M'$ be the subset of $M_a$ consisting of
the edges of $M_a$ that are contained in the multiset $E(\cal S)$.
By the proof of Lemma \ref{lemma1}, the set $M'$ consists of $|\mathcal S|< n$ distinct edges.
Each edge of $M_a\smallsetminus M'$ appears exactly $\sum_{F\in\mathcal F_1} s_F\leq\lambda_0$
times among the edges of the $1$-factors in $\mathcal S\cup (\cup_{F\in\mathcal F_1} F^H)$.
Each edge of $M'$ appears exactly $1+\sum_{F\in\mathcal F_1} s_F$
times among the edges of the $1$-factors in $\mathcal S\cup (\cup_{F\in\mathcal F_1} F^H)$.
Whence $\sum_{F\in\mathcal F_1} s_F<\lambda_0$, otherwise the edges of $M'$ would
appear at least $\lambda_0+1$ times among the edges of the $1$-factors in $\mathcal F_0$.
Since the edges of $M_a\smallsetminus M'$ appear $\sum_{F\in\mathcal F_1} s_F<\lambda_0$ times,
the $1$-factorization $\mathcal F_0$ must contain $\lambda_0-\sum_{F\in\mathcal F_1} s_F>0$
copies of $M_a$. Consequently, each edge of $M'$ appears at least $\lambda_0+1$ among the
edges of the $1$-factors in $\mathcal F_0$. That yields a contradiction. Hence, either $\mathcal F_0$
contains no $1$-factor of $F^H_i$ or $\mathcal F^H_i\subseteq\mathcal F_0$.
\end{proof}

\begin{proposition}\label{pro1}
Let $n\geq 5$ and $(n-2)/3\leq\lambda\leq n-2$ such that $n-\lambda$ is even.
There exists an indecomposable $1$-factorization of $\lambda K_{2n}$ which is not simple.
\end{proposition}

\begin{proof}
Identify $\lambda K_{2n}$ with  $\lambda K_G$. If $\lambda< n-2$, then $n>5$ and
we consider the $1$-factor $A$ in Figure \ref{fig2}(a).
For $\lambda=n-2$ we consider the $1$-factor $A$ in Figure \ref{fig3_AB}
with $\alpha=1$. If $\lambda < n-2$, then $A$ contains exactly
$(n-\lambda-2)/2$ edges of $M_1$ as well as
$(n-\lambda-2)/2$ edges of $M_{n-1}$. It also contains $\lambda$ edges of $M_0$,
one edge of $M_2$ and one edge of $M_{n-2}$. If $\lambda = n-2$, then
$A$ contains exactly $\lambda$ edges of $M_0$ as well as one edge of $M_1$
and one edge of $M_{n-1}$. In both cases the stabilizer of $A$ in $H$ is trivial and when $\lambda < n-2$,
the condition $(n-2)/3 \le \lambda$ assures that $(n-\lambda -2)/2 \le \lambda$.
Therefore ${\cal F}'=\{A\}$ satisfies Lemma \ref{lemma3} and a
$1-$factorization ${\cal F}$ of $\lambda K_G$ is constructed as prescribed.
We prove that ${\cal F}$ is indecomposable. Suppose that ${\cal F}_0\subseteq {\cal F}$
is a $1-$factorization of $\lambda_0 K_G$, $\lambda_0< \lambda$.
The $1-$factor $A$ satisfies condition (i) of Lemma \ref{lemma4}
(set $M_a= M_2$ or $M_a=M_1$ according to whether $\lambda < n-2$ or $\lambda = n-2$, respectively).
Therefore it is either $A^H\subset {\cal F}_0$ or $A^H\cap {\cal F}_0=\emptyset$.
In the former case, each edge of $M_0$ appears $\lambda$ times in the multiset $E({\cal F}_0)$,
that is, $\lambda = \lambda_0$, a contradiction. In the latter case, no edge of $M_0$ appears in
$E(\mathcal F_0)$, a contradiction.
\end{proof}


\begin{proposition}\label{pro2}
Let $n\geq 5$ and $(n+1)/3\leq\lambda\leq n-3$ such that $n-\lambda$ is odd.
There exists an indecomposable $1$-factorization of $\lambda K_{2n}$ which is not simple.
\end{proposition}

\begin{proof}
The proof is similar to the proof of Proposition \ref{pro1}. 
\end{proof}

\begin{figure*}
	\begin{center}
	\includegraphics[width=12cm]{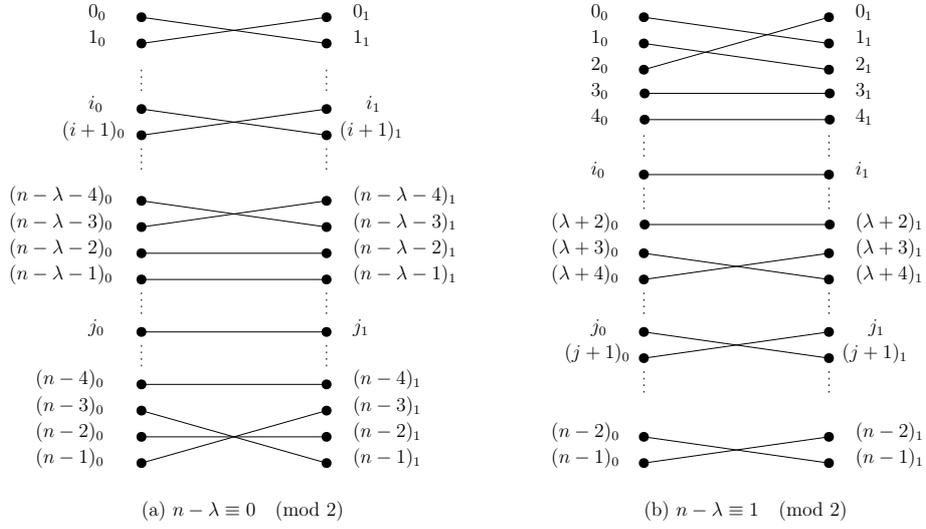}
\end{center}
	\caption{The $1$-factor $A$ in the case: (a) $n-\lambda$ even, $\lambda< n-2$; (b) $n-\lambda$ odd}
	\label{fig2}
\end{figure*}

\begin{proposition}\label{pro4}
Let $n\geq 7$ and $n-1\leq\lambda\leq n$. There exists an indecomposable
$1$-factorization of $\lambda K_{2n}$ which is not simple.
\end{proposition}

\begin{proof}
Identify $\lambda K_{2n}$ with  $\lambda K_G$ and set $\lambda=n-1+r$, where $0\leq r\leq 1$.
We consider the $1$-factors $A$ and $B_r$ in Figure \ref{fig3_AB}. In the definition of $A$, we set $\alpha=3$ if $r=0$;
$\alpha=2$ if $r=1$. The $1$-factors $A$, $B_r$ have trivial stabilizer in $H$.
Moreover, the multiset $E(A)\cup E(B_r)$ is contained in the multiset
$E(\mathcal M)$, where $\mathcal M=\{M_0, M_1, M_{\alpha}, M_{n-\alpha}, M_{r+2}\}$.
We note that the $1$-factors in $\mathcal M$ are pairwise distinct, since $n\geq 7$.
Whence $t(M_a)=|E(M_a)\cap A|+|E(M_a)\cap E(B_r)|\leq\lambda$ for every $M_a\in\mathcal M$.
More specifically, $t(M_0)=(n-2)+(r+1)=\lambda$, $t(M_1)=n-r-2=\lambda-1$,
$t(M_a)=1$ for every $a\in\{\alpha, n-\alpha, r+2\}$.
By Lemma \ref{lemma3}, we construct a $1$-factorization $\mathcal F$ of $\lambda K_G$
that contains $A^H\cup B^H_r$.

We prove that ${\cal F}$ is indecomposable. Firstly, note that if ${\cal F}_0\subseteq {\cal F}$
is a $1-$factorization of $\lambda_0 K_G$ , $\lambda_0<\lambda$, then $F^H\subset {\cal F}_0$ or
$F^H\cap {\cal F}_0=\emptyset$ for $F\in\{A, B_r\}$. This follows from Lemma \ref{lemma4} by observing
that $A$ and $M_{\alpha}$ satisfy condition $(i)$. The same can be repeated for $B_r$ and $M_{r+2}$.
If $A^H\subset {\cal F}_0$ and $B_r^H\subset {\cal F}_0$, then each edge of $M_0$ appears $\lambda$ times
in the multiset $E(\mathcal F_0)$ and then $\lambda_0=\lambda$, a contradiction.
In the same manner, if $A^H\cap {\cal F}_0= B_r^H \cap {\cal F}_0=\emptyset$, then no edge of $M_0$ appears
in the multiset $E(\mathcal F_0)$, a contradiction. Therefore, exactly one of the orbits $A^H$, $B^H_r$
is contained in $\mathcal F_0$. Without loss of generality, we can assume that $A^H\subset {\cal F}_0$ and
$B_r^H \cap {\cal F}_0=\emptyset$. Each edge  of $M_0$ appears at least $n-2$ times in the multiset $E({\cal F}_0)$,
that is, $\lambda_0\ge n-2$. Each edge of $M_1$ appears at least $n-2-r$ in the multiset
$E({\cal F}\smallsetminus {\cal F}_0)$, that is, $\lambda - \lambda_0 \ge n-2-r$. By summing up  these two relations,
we have $\lambda \ge 2n-4-r$ and since $\lambda \le n$, this yields $n\le 5$, a contradiction.
\end{proof}

\begin{figure*}
	\begin{center}
	\includegraphics[width=13cm]{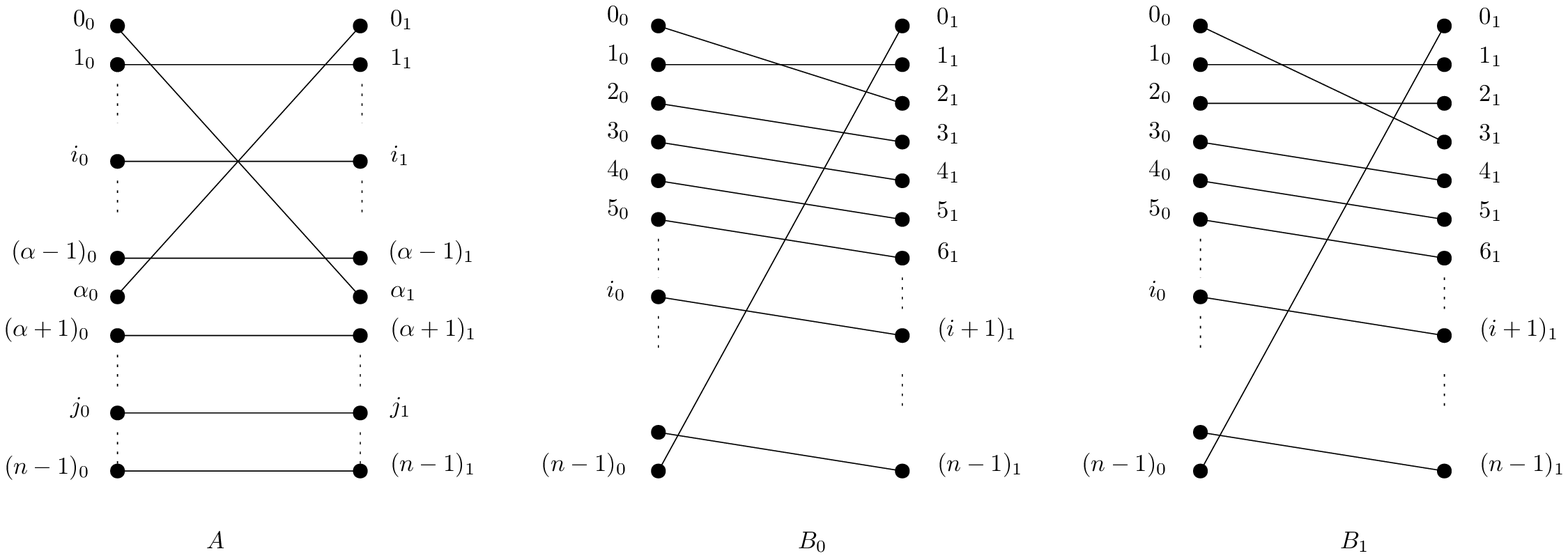}
\end{center}
	\caption{The $1$-factors $A$ and $B_r$, $r=0, 1$, defined in the proof of Proposition \protect\ref{pro4}.}
	\label{fig3_AB}
\end{figure*}

\begin{proposition}\label{pro3}
Let $n\geq 9$ and $n+1\leq\lambda\leq 2n-8$. There exists an indecomposable
$1$-factorization of $\lambda K_{2n}$ which is not simple.
\end{proposition}

\begin{proof} Identify $\lambda K_{2n}$ with  $\lambda K_G$.
We distinguish the cases $n\neq 11$ and $n=11$.
For $n\neq 11$, we set $\lambda=n+r$, where $1\leq r\leq n-8$,
and consider the $1$-factors $A$ and  $B=B_0$ in Figure \ref{fig3_AB}. In the definition of $A$
we set $\alpha=3$. We also define the $1$-factors $C$ and $D$ in Figure \ref{fig4_CD}.

For $n=11$, we set $\lambda=9+r$, where $3\leq r\leq 5$.
We consider the $1$-factor $A$ in Figure \ref{fig3_AB}, where $\alpha=2$ or $\alpha=3$,
according to whether $r=3, 4$ or $r=5$, respectively.
For $r=3, 4$ we also consider the $1$-factor $B=\{[i_0, i_1]: 1\leq i\leq r\}\cup$
$\{[i_0, (i+1)_1]: r+1\leq i\leq 10\}\cup$$\{[0_0, (r+1)_1]\}$.
For $r=5$, we consider the $1$-factor $B=B_0$ in Figure \ref{fig3_AB} and
the $1$-factor $C=\{[i_0, i_1]: 1\leq i\leq 4\}\cup$$\{[i_0, (i+1)_1]: 5\leq i\leq 10, i\neq 6\}\cup$
$\{[0_0, 7_1], [6_0, 5_1]\}$. We can construct a $1$-factorization $\mathcal F$ of 
$\lambda K_G$  as described in Lemma \ref{lemma3}.
By Lemma \ref{lemma4}, the $1$-factorization $\mathcal F$ is indecomposable.
The proof is similar to that of Proposition \ref{pro4}
\end{proof}

\begin{proposition}\label{pro7}
Let $n\geq 9$ and $\lambda=2n-7$. There exists an indecomposable
$1$-factorization of $\lambda K_{2n}$ which is not simple.
\end{proposition}

\begin{proof} We set $\lambda=n+r$ with $r=n-7$ and consider the
$1$-factors in $\mathcal F'=\{A, B, C, D\}$, where $A$ and  $B=B_0$ are described in Figure \ref{fig3_AB}.
In the definition of $A$ we set $\alpha=3$. The $1$-factors $C$ and $D$ are defined in Figure \ref{fig4_CD}.
The assertion follows from Lemma \ref{lemma4}.
\end{proof}

\begin{figure*}
	\begin{center}
	\includegraphics[width=10cm]{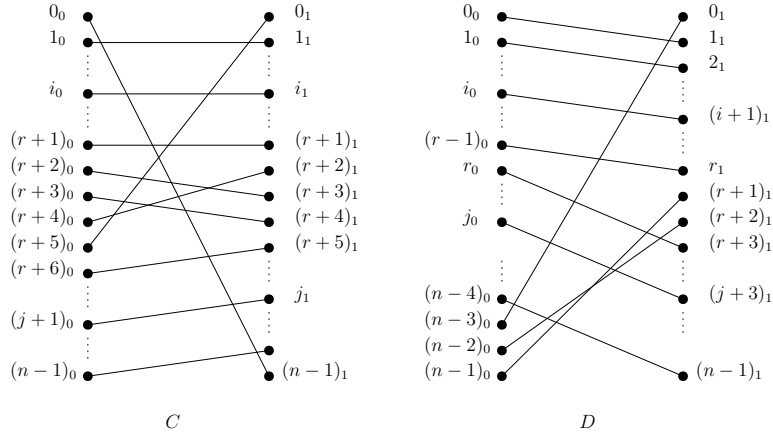}
\end{center}
	\caption{The $1$-factors $C$ and $D$ defined in the proof of Proposition \protect\ref{pro3}.}
	\label{fig4_CD}
\end{figure*}


\begin{proposition}\label{pro5}
Let $n\geq 9$ and $2n-6\leq\lambda\leq 2n-3$. There exists an indecomposable
$1$-factorization of $\lambda K_{2n}$ which is not simple.
\end{proposition}

\begin{proof} Identify $\lambda K_{2n}$ with  $\lambda K_G$ and set $\lambda=2n-r$, where
$3\leq r\leq 6$. We consider the $1$-factors $A$ and $B=B_1$ in Figure \ref{fig3_AB}.
In the definition of the $1$-factor $A$, the parameter $\alpha$ assumes the value $\alpha=2$ if $r\in\{3, 5, 6\}$;
$\alpha=4$ if $r=4$. We define the $1$-factor $C$ as in Figure \ref{fig6_AC}.
We also consider the $1$-factor $D_r$ in Figure \ref{fig7_D} for $r=3, 4$
and in Figure \ref{fig8_D} for $r=5, 6$.
We can apply Lemma \ref{lemma3} and construct a $1$-factorization $\mathcal F$ of $\lambda K_G$
as prescribed. By Lemma \ref{lemma4}, we can prove that $\mathcal F$ is indecomposable.
\end{proof}

\begin{figure*}
	\begin{center}
	\includegraphics[width=10cm]{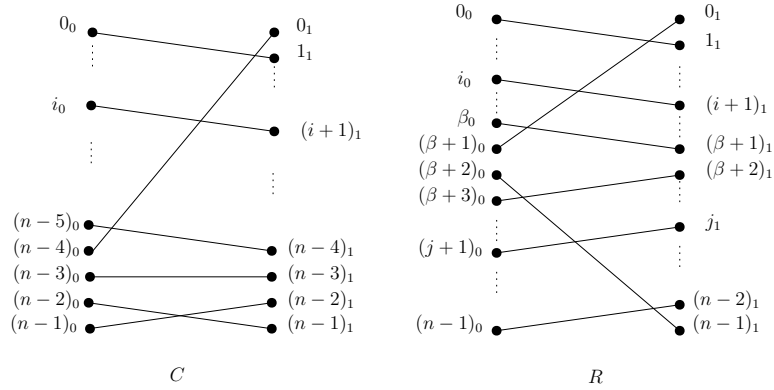}
\end{center}
	\caption{The $1$-factors $C$ and $R$ defined in the proof of Proposition \protect\ref{pro5} and \protect\ref{pro8},
	respectively.}
	\label{fig6_AC}
\end{figure*}

\begin{figure*}
	\begin{center}
	\includegraphics[width=10cm]{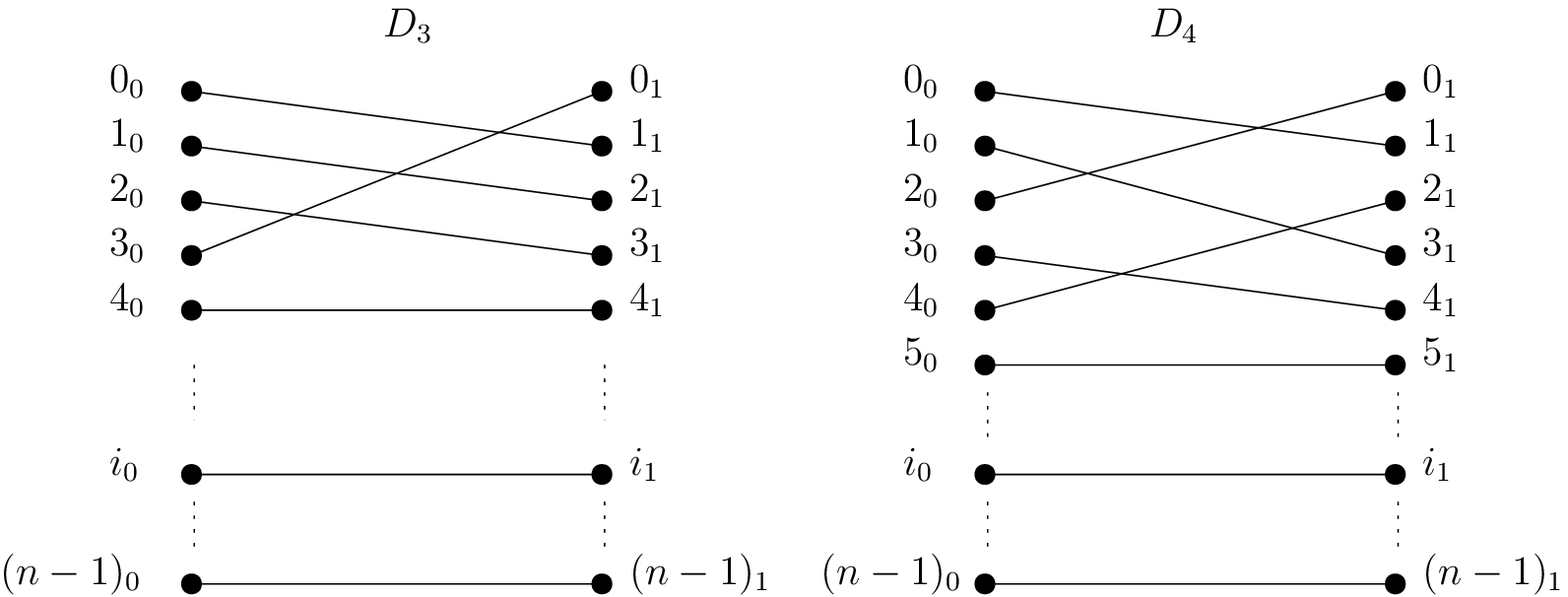}
\end{center}
	\caption{The $1$-factor $D_r$, $r=3, 4$, defined in the proof of Proposition \protect\ref{pro5}.}
	\label{fig7_D}
\end{figure*}

\begin{figure*}
	\begin{center}
	\includegraphics[width=10cm]{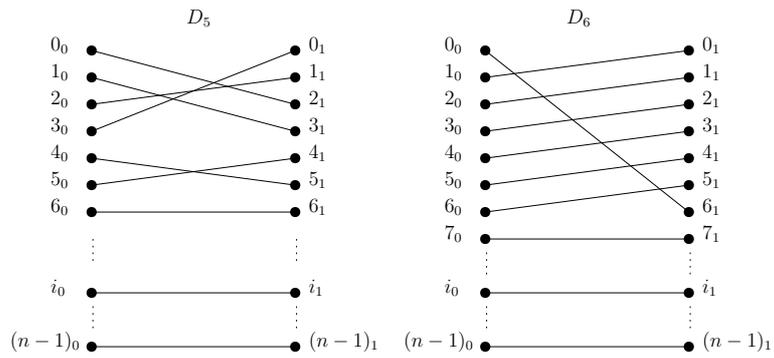}
\end{center}
	\caption{The $1$-factor $D_r$, $r=5, 6$, defined in the proof of Proposition \protect\ref{pro5}.}
	\label{fig8_D}
\end{figure*}


\begin{proposition}\label{pro6}
Let $n\geq 9$ and $\lambda=2n-2$. There exists an indecomposable
$1$-factorization of $\lambda K_{2n}$ which is not simple.
\end{proposition}

\begin{proof} Identify $\lambda K_{2n}$ with  $\lambda K_G$.
We distinguish the cases $n\geq 11$ and $n=9, 10$.
For $n\geq 11$ we consider the $1$-factor $A$ in Figure \ref{fig3_AB} with
$\alpha=2$ and the $1$-factor $B_1=D$. We also consider
the $1$-factors $B$, $C$ in Figure \ref{fig9_BCD}.

For $n=9$, $10$, we consider two copies of the $1$-factor $A$ in Figure \ref{fig3_AB}.
We denote by $A$ the copy with $\alpha=2$ and by $B$ the copy with $\alpha=3$ or $4$, according to whether $n=10$ or
$n=9$, respectively. We consider the $1$-factors $C$, $D$ and $R_n$,
where $C=\{[i_0, (i+1)_1]: 2\leq i\leq n-1\}$$\cup\{[0_0, 2_1], [1_0, 1_1]\}$;
$D=\{[i_0, (i+1)_1]: 2\leq i\leq n-3\}$$\cup\{[0_0, (n-1)_1], [(n-1)_0, 0_1], [(n-2)_0, 2_1]$, $[1_0, 1_1]\}$.
$R_9=\{[i_0, (i+1)_1]: 0\leq i\leq 2\}$$\cup\{[i_0, (i+2)_1]: 3\leq i\leq 7\}$
$\cup\{[8_0, 4_1]\}$; $R_{10}=\{[i_0, (i+1)_1]: 0\leq i\leq 2\}$$\cup\{[i_0, (i+2)_1]: 3\leq i\leq 7\}$
$\cup\{[8_0, 0_1], [9_0, 4_1]\}$. We can construct a $1$-factorization
$\mathcal F$ as described in Lemma \ref{lemma3}. By Lemma \ref{lemma4},
we can prove that $\mathcal F$ is indecomposable.
\end{proof}

\begin{proposition}\label{pro8}
Let $n\geq 9$ and $2n-1\leq\lambda\leq 2n$. There exists an indecomposable
$1$-factorization of $\lambda K_{2n}$ which is not simple.
\end{proposition}

\begin{proof} Identify $\lambda K_{2n}$ with $\lambda K_G$. We consider two copies of the $1$-factor $A$
in Figure \ref{fig3_AB}. We denote by $A$ the copy with $\alpha=2$ ($\alpha=4$ if $n=9$ and $\lambda=18$)
and by $B$ the copy with $\alpha=3$. We also consider the $1$-factors
$C$, $D$, $R$. For $n\geq 9$ and $(n, \lambda)\neq (9, 18)$, the $1$-factor $C$ corresponds to the $1$-factor $B_1$
in Figure \ref{fig3_AB}. For $(n, \lambda)=(9, 18)$ it corresponds to the $1$-factor $C$ in Figure \ref{fig9_BCD}.
For $n\geq 9$ and $\lambda=2n-1$, the $1$-factor $D$ corresponds to the $1$-factor $B_0$ in Figure \ref{fig3_AB}.
For $n>9$ and $\lambda=2n$, the $1$-factor $D$ is defined in Figure \ref{fig9_BCD}.
For $n=9$ and $\lambda=2n$, it corresponds to the $1$-factor $B_0$ in Figure \ref{fig3_AB}.
For $n\geq 9$ and $(n, \lambda)\neq (9, 18)$, the $1$-factor $R$ is defined in
Figure \ref{fig6_AC}. In the definition of $R$ we set $\beta=3$ or $\beta=4$
according to whether $\lambda=2n-1$ or $\lambda=2n$, respectively ($\beta=5$ if $n=10$ and $\lambda=2n$).
For $(n, \lambda)=(9, 18)$, we set $R=\{[i_0, (i+1)_1]: 0\leq i\leq 4, i=8\}$$\cup\{[i_0, (i+2)_1]: 5\leq i\leq 6\}$
$\cup\{[7_0, 6_1]\}$. We construct a $1$-factorization $\mathcal F$ of $\lambda K_G$ as described in Lemma \ref{lemma3}.
By Lemma \ref{lemma4}, we can prove that $\mathcal F$ is indecomposable.
\end{proof}

\begin{figure*}
	\begin{center}
	\includegraphics[width=12cm]{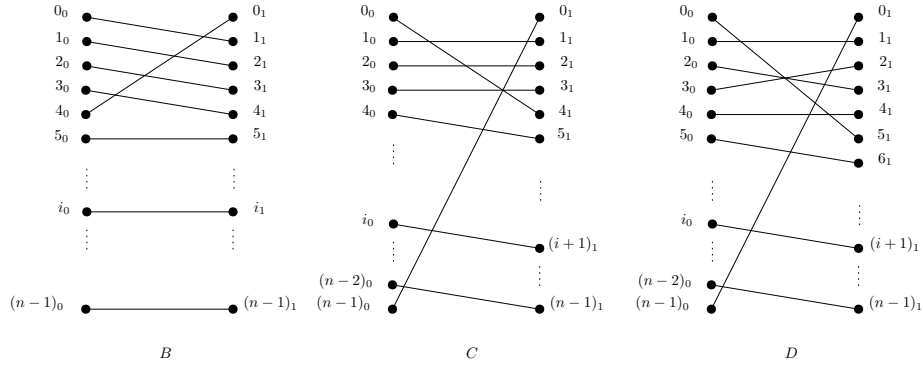}
\end{center}
	\caption{The $1$-factors $B$, $C$ defined in the proof of Proposition \protect\ref{pro6}
	and the $1$-factor $D$ defined in the proof of Proposition \protect\ref{pro8} for $\lambda=2n$}
	\label{fig9_BCD}
\end{figure*}


Combining the constructions in the previous propositions, the following result holds.

\begin{theorem}\label{th1}
Let $n\geq 9$. For every $(n-2)/3\leq\lambda\leq 2n$ there exists an indecomposable
$1$-factorization of $\lambda K_{2n}$ which is not simple.\qed
\end{theorem}

\section{Simple and indecomposable $1$-factorizations.}\label{sec:IOF_simple}

In this section we use Theorem \ref{th1} and
Corollary $4.1$ in \cite{CCR} to find simple
and indecomposable $1$-factorizations of $\lambda K_{2n}$.
We also generalize the result in \cite{CCR} about the
existence of simple and indecomposable $1$-factorizations of $\lambda K_{2n}$,
where $2n-1$ is a prime and $\lambda=(n-1)/2$. We recall the statement of Corollary $4.1$ .

\begin{cor4.1}\label{cor4.1}\cite{CCR}
If there exists an indecomposable $1$-factorization of $\lambda K_{2n}$
with $\lambda\leq 2n-1$, then there exists a simple and indecomposable
$1$-factorization of $\lambda K_{2s}$ for $s\geq 2n$.
\end{cor4.1}

The following results hold.

\begin{theorem}\label{th2}
Let $s\geq 18$. For every $2\leq\lambda\leq 2\lfloor s/2\rfloor-1$ there exists a
simple and indecomposable $1$-factorization of $\lambda K_{2s}$.
\end{theorem}

\begin{proof} For every $n\geq 9$ we set $I_n=\{\lambda\in\mathbb Z: (n-2)/3\leq\lambda\leq 2n-1\}$
and note that $I_n\cup I_{n+1}=\{\lambda\in\mathbb Z: (n-2)/3\leq\lambda\leq 2(n+1)-1\}$.
Consider $s\geq 2n\geq 2\cdot 9$. By Corollary $4.1$ of \cite{CCR}, for every $\lambda\in I_n$ there exists a
simple and indecomposable $1$-factorization of $\lambda K_{2s}$.
Since we can consider $9\leq n\leq\lfloor s/2\rfloor$, we obtain a simple and indecomposable
$1$-factorization of $\lambda K_{2s}$ for every $\lambda\in\cup^{\lfloor s/2\rfloor}_{n=9}\, I_n=\{\lambda\in\mathbb Z: 7/3\leq\lambda\leq 2\lfloor s/2\rfloor-1\}$. Since $s\geq 2\cdot 5$, from Proposition \ref{pro2} and Corollary $4.1$ we also obtain
a simple and indecomposable $1$-factorization of $\lambda K_{2s}$ for $\lambda=2$. Hence the assertion follows.
\end{proof}

\begin{theorem}\label{th3}
Let $2n-1$ be a prime power and let $\lambda=n-1$.
There exists a simple and indecomposable $1$-factorization of $\lambda K_{2n}$.
\end{theorem}

\begin{proof} Let $2n-1=p^m$, with $p$ an odd prime and $m\ge 1$. Let $GF(p^m)$ be
the Galois field of order $p^m$ and let $v$ be a generator of the
cyclic multiplicative group $GF(p^m)^*=GF(p^m)-\{0\}$. It is well known that
$v$ is a root of an irriducible polynomial over $\mathbb Z_p$ of degree $m$, the field
$GF(p^m)$ is an algebraic extension of $\mathbb Z_p$ and it is
$GF(p^m)=\mathbb Z_p(v)=\{a_0+a_1v+a_2v^2+\dots +a_{m-1}v^{m-1} \ | \ a_i\in \mathbb Z_p\}$.
Let $V=GF(p^m)\cup \{\infty \}$, $\infty \notin GF(p^m)$, and identify the vertices of the complete multigraph $(n-1)K_{2n}$ with the elements of $V$, thus the edges are in the multiset $(n-1)\binom{V}{2}$.
The affine linear group $AGL(1,p^m) = \{\phi_{b,a}: a, b \in GF(p^m), b\ne 0\}$ is a permutation group on $V$ where each $\phi_{b,a}$ fixes $\infty$ and maps $x\in V\smallsetminus \{\infty\}$ onto $xb + a$. This action extends to edges and $1$-factors. For each edge $e=[x,y]$ and for each $1-$factor $F$,  we set $e^{\phi_{b,a}}=eb+a = [xb+a,yb+a]$ and $F^{\phi_{b,a}}=Fb+a$.

If $x\ne \infty$ and $y\ne \infty$ we call $\partial e = \{\pm (y-x)\}$ the {\it difference set} of $e$.

Consider the following set of edges:

\vskip0.2truecm\noindent
$A_0=\{[(2i-1)+a_1v+a_2v^2+ \dots + a_{m-1}v^{m-1}, 2i+a_1v+\dots + a_{m-1}v^{m-1}],$

$\ \ \ \ 1\le i \le (p-1)/2, a_r \in \mathbb Z_p, 1\leq r\leq m-1 \}$

\vskip0.2truecm\noindent
$A_1=\{[(2i-1)v+a_2v^2+ \dots + a_{m-1}v^{m-1}, (2i)v+a_2v^2+\dots + a_{m-1}v^{m-1}],$

$\ \ \ \ \ 1\le i \le (p-1)/2, \ a_r \in \mathbb Z_p, 2\leq r\leq m-1 \}$

\vskip0.2truecm\noindent
$A_2=\{[(2i-1)v^2+ \dots + a_{m-1}v^{m-1}, (2i)v^2+\dots + a_{m-1}v^{m-1}],$

$\ \ \ \  \ 1\le i \le (p-1)/2, \ a_r \in \mathbb Z_p, 3\leq r\leq m-1 \}$

\vskip0.3truecm\noindent
$\dots$

\noindent
$A_{m-1}=\{[(2i-1)v^{m-1}, (2i)v^{m-1}], \ 1\le i \le (p-1)/2 \}$.

\vskip 0.3truecm
\noindent
Obviously if $m=1$ we just have $\mathbb Z_p(v)=\mathbb Z_p$ and we just take the set $A_0$.
\vskip0.3truecm
\noindent
Observe that each set $A_j$, $j=0,\dots, m-1$, contains exactly $p^{m-j-1}(p-1)/2$ edges
with difference set $\{\pm v^j\}$. Let $F$ be the $1-$factor given by: $\{[0,\infty]\}\cup A_0\cup A_1 \cup \dots \cup A_{m-1}$.
The set ${\cal F}=F^{AGL(1,p^m)}$ is a simple and indecomposable $1-$factorization of $(n-1)K_{2n}$.
\end{proof}

\section{Conclusions.}\label{sec:final}

Our methods of construction can be used to obtain indecomposable
$1$-factorizations of $\lambda K_{2n}$ for some values of $\lambda >2n$. These $1$-factorizations
are not simple and do not provide simple $1$-factorizations, since for these values of $\lambda$
we cannot apply Corollary $4.1$ of \cite{CCR}. 

As remarked in Section \ref{sec:intro}, a necessary condition for the
existence of an indecomposable $1$-factorization of $\lambda K_{2n}$ is
$\lambda<[n(2n-1)]^{n(2n-1)}\binom{2n^3+n^2-n+1}{2n^2-n}$.
It would be interesting to know whether for every $n\geq 4$ 
there exists a parameter $\lambda(n)<[n(2n-1)]^{n(2n-1)}\binom{2n^3+n^2-n+1}{2n^2-n}$ depending from $n$
such that for every $\lambda>\lambda(n)$ there is no indecomposable $1$-factorization
of $\lambda K_{2n}$.

\end{document}